\documentclass{article}

\usepackage{lineno,hyperref}
\usepackage{amsmath}
\usepackage{amssymb}
\usepackage{amsthm}

\newtheorem{lemma}{Lemma}
\newtheorem{theorem}{Theorem}

\DeclareMathOperator\mat{M}
\DeclareMathOperator\elin{E}
\DeclareMathOperator\glin{GL}
\DeclareMathOperator\stlin{St}
\DeclareMathOperator\Ker{Ker}
\newcommand{\op}{{\mathrm{op}}}
\DeclareMathOperator{\Aut}{Aut}
\DeclareMathOperator{\End}{End}

\newcommand{\sub}[2]{{_{#1\!}{#2}}}
\newcommand{\up}[2]{{^{#1}\!{#2}}}

\title{Groups with \(\mathsf A_\ell\)-commutator relations}

\author{
  Egor Voronetsky\thanks{The work was supported by the Theoretical Physics and Mathematics Advancement Foundation ``BASIS'' and by ``Native towns'', a social investment program of PJSC ``Gazprom Neft''.} \\
  Chebyshev Laboratory, \\
  St. Petersburg State University, \\
  14th Line V.O., 29B, \\
  Saint Petersburg 199178 Russia \\
}

\begin{document}
\maketitle

\begin{abstract}
If \(A\) is a unital associative ring and \(\ell \geq 2\), then the general linear group \(\glin(\ell, A)\) has root subgroups \(U_\alpha\) and Weyl elements \(n_\alpha\) for \(\alpha\) from the root system of type \(\mathsf A_{\ell - 1}\). Conversely, if an arbitrary group has such root subgroups and Weyl elements for \(\ell \geq 4\) satisfying natural conditions, then there is a way to recover the ring \(A\). We prove a generalization of this result not using the Weyl elements, so instead of the matrix ring \(\mat(\ell, A)\) we construct a non-unital associative ring with a well-behaved Peirce decomposition.
\end{abstract}

\section{Introduction}

General linear groups \(\glin(\ell, A)\), their elementary subgroups \(\elin(\ell, A)\) and the Steinberg groups \(\stlin(\ell, A)\) over a unital associative ring \(A\) are a classical subject of lower unstable algebraic \(\mathrm K\)-theory, see \cite{hahn-omeara} and \cite{milnor}. For example, it is known that \(\elin(\ell, A)\) and \(\stlin(\ell, A)\) are perfect groups for \(\ell \geq 3\), \(\stlin(\ell, A)\) is centrally closed for \(\ell \geq 5\). If \(A\) is a finite \(K\)-algebra for a unital commutative ring \(K\), then \(\elin(\ell, A)\) is normal in \(\glin(\ell, A)\) for \(\ell \geq 3\) and \(\stlin(\ell, A)\) is a central extension of \(\elin(\ell, A)\) for \(\ell \geq 4\).

These groups have naturally defined root subgroups \(U_\alpha\) with explicit isomorphisms \(x_\alpha \colon A \to U_\alpha\), where \(\alpha\) are the roots of the root system \(\Phi\) of type \(\mathsf A_{\ell - 1}\). The maps \(x_\alpha\) satisfy the Steinberg relations
\begin{align*}
x_\alpha(a + b) &= x_\alpha(a)\, x_\alpha(b); \\
[x_\alpha(a), x_\beta(b)] &= 1 \text{ if } \alpha + \beta \notin \Phi \cup \{0\}; \\
[x_\alpha(a), x_\beta(b)] &= x_{\alpha + \beta}(ab) \text{ or } x_{\alpha + \beta}(- ba) \text{ if } \alpha + \beta \in \Phi. 
\end{align*}
Notice that there is no commutator relation for \([x_\alpha(a), x_{-\alpha}(b)]\). Also, the Weyl elements \(n_\alpha = x_\alpha(1)\, x_{-\alpha}(-1)\, x_\alpha(1)\) act on the root subgroups by
\[\up{n_\alpha}{x_\beta(a)} = x_{s_\alpha(\beta)}(\pm a),\]
where \(s_\alpha \colon \Phi \to \Phi,\, \beta \mapsto \beta - 2 \frac{(\alpha, \beta)}{(\alpha, \alpha)} \alpha\) is a reflection from the Weyl group of \(\Phi\).

Conversely, if \(G\) is an arbitrary group with root subgroups \(U_\alpha\) and the distinguished elements \(n_\alpha\) parametrized by the roots of \(\Phi\) and \(\ell \geq 4\), then under suitable assumptions it is possible to recover the ring \(A\) and a homomorphism \(\stlin(\ell, A) \to G\). This is proved in \cite{graded-groups}, as well as a generalization for the Chevalley groups of types \(\mathsf D_\ell\) and \(\mathsf E_\ell\). Much more general results involving arbitrary root systems are known in the case of division rings, see e.g. \cite{root-group-datum} and \cite{abs-root-subgr}.

There is a generalization of the groups \(\glin(\ell, A)\), \(\elin(\ell, A)\), and \(\stlin(\ell, A)\) for the non-matrix case. If \(R\) is a unital associative ring with a complete family of full idempotents \(e_1\), \ldots, \(e_\ell\) (i.e. such that \(R e_i R = R\)), then we have the Peirce decomposition \(R = \bigoplus_{1 \leq i, j \leq \ell} R_{ij}\), where \(R_{ij} = e_i R e_j\). Let \(\glin(R) = R^*\), \(t_{ij}(a) = 1 + a \in \glin(R)\) for \(i \neq j\) and \(a \in R_{ij}\), and the elementary subgroup \(\elin(R) \leq \glin(R)\) be the subgroup generated by \(t_{ij}(a)\). The Steinberg group \(\stlin(R)\) is defined as the abstract group generated by \(x_{ij}(a)\) for \(i \neq j\) and \(a \in R_{ij}\) satisfying the appropriate analogues of the Steinberg relations. If \(R = \mat(\ell, A)\) is the matrix ring and \(e_i = e_{ii}\) are the diagonal idempotents, then this coincides with the classical definitions. Some results of \(\mathrm K\)-theory still hold in this generality, for example, the centrality of the extension \(\stlin(R) \to \elin(R)\) is proved in \cite{central-k2}. But there is no analogue of \(n_\alpha\) in \(\glin(R)\), since it is possible that \(R\) contains no invertible elements in \(e_1 R e_2 + e_2 R e_1\) for \(\ell = 2\) (say, if \(R = \mat(3, \mathbb R)\), \(e_1 = e_{11}\), and \(e_2 = e_{22} + e_{33}\)).

In this generality the collection of the root subgroups \(U_\alpha \leq \glin(R)\) satisfies the axioms of \(\Phi\)-commutator relations from \cite{st-jordan} (if we omit the condition that the root subgroups generate the whole group). In this paper we give the necessary and sufficient conditions for \(\Phi\)-commutator relations to be induced from a Steinberg group over an associative ring with a Peirce decomposition, where \(\ell \geq 4\).

Actually, in our main result we do not require that \(R\) is unital. Since we cannot consider the complete families of idempotents \(e_i \in R\), we have to consider rings with abstract Peirce decompositions \(R = \bigoplus_{1 \leq i, j \leq \ell} R_{ij}\). Our main result has two variants, depending on a generalization of the fullness of the idempotents \(e_i\) to the non-unital case. Namely, we require that \(R_{ij} R_{jk} = R_{ik}\) for all \(i\), \(j\), \(k\), this is the sufficient condition for \(\elin(R)\) and \(\stlin(R)\) to be perfect. Also, we require that either \(R \otimes_R R \to R\) is an isomorphism (i.e. \(R\) is firm) or \(R \to \End(R)^\op \times \End(R)\) is injective. In the first case the main results of \cite{central-k2} hold with the same proofs, i.e. the root elimination and the centrality of \(\mathrm K_2\). In the second case the root subgroups may be defined in \(\Aut(R)\) instead of \(R\) by lemma \ref{center-perf} proved below, i.e. we may consider a generalized projective general linear group with \(\Phi\)-commutator relations.

The firm rings seems to be the most natural generalization of unital rings, see e.g. \cite{tensor-idem} and \cite{quillen}. For example, the ring of finitary matrices \(\mat(\infty, A)\) over a unital associative ring \(A\) is firm, but not unital. There are finite firm non-unital algebras over fields, see \cite[example 5]{locally-unital}.

\section{Peirce decompositions of non-unital rings}

All rings in this paper are associative, but not necessarily unital. For a ring \(R\) and its non-unital modules \(M_R\), \(\sub RN\) we use the notation \(M \otimes_R N = (M \otimes N) / \langle mr \otimes n - m \otimes rn \mid m \in M,\, n \in N,\, r \in R \rangle\), where the unlabelled tensor product is taken over \(\mathbb Z\). A module \(M_R\) is called
\begin{itemize}
\item \textit{unital}, if \(MR = M\);
\item \textit{firm}, if \(M \otimes_R R \to M\) is an isomorphism;
\item \textit{reduced}, if it is unital and there are no non-zero \(m \in M\) such that \(mR = 0\);
\end{itemize}
and similarly for left modules, see \cite{quillen} for details. A ring \(R\) is called
\begin{itemize}
\item \textit{idempotent}, if \(R^2 = R\);
\item \textit{firm}, if \(R \otimes_R R \to R\) is an isomorphism;
\item \textit{reduced}, if it is idempotent and there are no non-zero \(x \in R\) such that \(Rx = xR = 0\).
\end{itemize}
Any unital ring \(R\) is firm and reduced, all unital modules over it are also firm and reduced. Moreover, a right module \(M\) over a unital ring is unital in the above terminology if and only if \(m 1 = m\) for all \(m \in M\).

We say that a ring \(R\) has a \textit{Peirce decomposition} of rank \(\ell \geq 0\) if
\begin{align*}
R &= \bigoplus_{1 \leq i, j \leq \ell} R_{ij}; \\
R_{ij} R_{kl} &= 0 \text{ for } j \neq k; \\
R_{ij} R_{jk} &\leq R_{ik}.
\end{align*}
For example, if \(R\) is unital with a complete family of orthogonal idempotents \(e_1\), \ldots, \(e_\ell\), then \(R_{ij} = e_i R e_j\) is a Peirce decomposition of \(R\). It is easy to see that every Peirce decomposition of a unital ring is of this type. A Peirce decomposition of a \(K\)-algebra \(R\), where \(K\) is a unital commutative ring, is a Peirce decomposition in the above sense such that \(R_{ij}\) are \(K\)-submodules.

Generalizing the properties of rings, we say that a Peirce decomposition of \(R\) is
\begin{itemize}
\item \textit{idempotent}, if \(R_{ij} R_{jk} = R_{ik}\) for all \(i\), \(j\), \(k\);
\item \textit{firm}, if \(R_{ij} \otimes_{R_{jj}} R_{jk} \to R_{ik}\) are isomorphisms for all \(i\), \(j\), \(k\);
\item \textit{reduced}, if it is idempotent and \(R\) is reduced.
\end{itemize}

\begin{lemma} \label{full-idem}
Let \(R\) be a unital ring with a complete family of orthogonal idempotents \(e_1\), \ldots, \(e_\ell\) and \(R_{ij} = e_i R e_j\) be the induced Peirce decomposition. Then the following properties are equivalent:
\begin{itemize}
\item the Peirce decomposition is idempotent;
\item the Peirce decomposition is firm;
\item the Peirce decomposition is reduced;
\item the idempotents \(e_i\) are full, i.e. \(R e_i R = R\).
\end{itemize}
\end{lemma}
\begin{proof}
If the Peirce decomposition is idempotent, then \(R_{ik} = R_{ij} R_{jk} \leq R e_j R\) for all \(i\), \(j\), \(k\), hence the idempotents are full. Conversely, suppose that the idempotents are full. Fix indices \(i\), \(j\), \(k\) and let \(1 = \sum_{t \in T} a_t e_j b_t\). If \(x \in R_{ik}\), then
\[x = \sum_{t \in T} (x a_t e_j) (e_j b_t e_k) \in R_{ij} R_{jk}.\]
If \(\sum_{s \in S} x_s y_s = 0\) for \(x_s \in R_{ij}\) and \(y_s \in R_{jk}\), then
\[\sum_{s \in S} x_s \otimes y_s = \sum_{s \in S} \sum_{t \in T} (e_i a_t e_j b_t x_s \otimes y_s - e_i a_t e_j \otimes e_j b_t x_s y_s).\]
In other words, the Peirce decomposition is firm. Clearly, it is also reduced.
\end{proof}

It turns out that the properties of a Peirce decomposition imply the corresponding properties of the ring.
\begin{lemma} \label{root-elim}
Let \(R\) be a ring with a firm Peirce decomposition of rank \(\ell \geq 2\). Let
\begin{align*}
R_{0i} &= R_{\ell - 1, i} \oplus R_{\ell i} \text{ for } 1 \leq i \leq \ell; \\
R_{j0} &= R_{j, \ell - 1} \oplus R_{j \ell} \text{ for } 0 \leq j \leq \ell.
\end{align*}
Then the Peirce decomposition of rank \(\ell - 1\) given by \(R = \bigoplus_{0 \leq i, j \leq \ell - 2} R_{ij}\) is firm. The same claim holds for idempotent Peirce decompositions.
\end{lemma}
\begin{proof}
Clearly, \(R_{ij} \otimes_{R_{jj}} R_{jk} \to R_{ik}\) are isomorphisms for \(1 \leq j \leq \ell\) and all \(i\), \(k\). In order to prove that \(R_{i0} \otimes_{R_{00}} R_{0j} \to R_{ij}\) are isomorphisms for all \(i\) and \(j\), it suffices to consider the case \(1 \leq i, j \leq \ell\) and to prove that
\begin{align*}
R_{i, \ell - 1} \otimes R_{\ell - 1, j} &\leq R_{i \ell} \otimes R_{\ell j} + \langle xy \otimes z - x \otimes yz \mid x \in R_{i, \ell - 1},\, y \in R_{\ell - 1, \ell},\, z \in R_{\ell j} \rangle; \\
R_{i, \ell - 1} \otimes R_{\ell j} &\leq \langle xy \otimes z - x \otimes yz \mid x \in R_{i \ell - 1},\, y \in R_{\ell \ell},\, z \in R_{\ell j} \rangle; \\
R_{i \ell} \otimes R_{\ell - 1, j} &\leq \langle xy \otimes z - x \otimes yz \mid x \in R_{i \ell},\, y \in R_{\ell \ell},\, z \in R_{\ell - 1, j} \rangle.
\end{align*}
But these relations follow from \(R_{i k} = R_{ij} R_{jk}\) for \(1 \leq i, j, k \leq \ell\). The idempotent case is clear.
\end{proof}

We also need a result from non-unital Morita theory.

\begin{lemma} \label{morita}
Let \(R\) be a firm ring, \(P_R\) and \(\sub RQ\) be firm modules, \(\langle -, = \rangle \colon Q \times P \to R\) be an \(R\)-bilinear map such that \(R = \langle Q, P \rangle\). Let also \(S = P \otimes_R Q\). Then \(\bigl(\begin{smallmatrix} S & P \\ Q & R \end{smallmatrix}\bigr)\) is an associative ring with a firm Peirce decomposition.
\end{lemma}
\begin{proof}
It may be directly checked that the canonical multiplication on the matrix ring is associative. Let \(\widetilde R = Q \otimes_S P\), then \(\widetilde R\) has an associative multiplication
\[(q \otimes p) (q' \otimes p') = \langle q, p \rangle q' \otimes p' = q \otimes p \langle q', p' \rangle.\]
Let \(\pi \colon \widetilde R \to R,\, q \otimes p \mapsto \langle q, p \rangle\) be the canonical homomorphism and \(I = \Ker(\pi)\). It is easy to see that \(I \widetilde R = \widetilde RI = 0\) and the ring \(\widetilde R\) is idempotent. We may consider \(\widetilde R\) and \(I\) as non-unital bimodules over \(R\).

The kernel of \(\widetilde R \otimes_R \widetilde R \to R \otimes_R R \cong R\) is the image of \((\widetilde R \otimes_R I) \oplus (I \otimes_R \widetilde R)\). But the latter group is zero since \(IR = RI = 0\) and \(\widetilde R = \widetilde R R = R \widetilde R\). It follows that the composition \(\widetilde R \otimes_R \widetilde R \to \widetilde R \to R\) is bijective, i.e. \(I = 0\).

Now it is easy to see that
\begin{align*}
S \otimes_S P &\cong P \otimes_R Q \otimes_S P \cong P \otimes_R R \cong P; \\
Q \otimes_S S &\cong Q \otimes_S P \otimes_R Q \cong R \otimes_R Q \cong Q; \\
S \otimes_S S &\cong P \otimes_R Q \otimes_S S \cong P \otimes_R Q \cong S.
\end{align*}
In other words, the Peirce decomposition is firm.
\end{proof}

In the next lemma we use the notation \(R_{i*} = \sum_{j = 1}^\ell R_{ij}\), \(R_{*j} = \sum_{i = 1}^\ell R_{ij}\).

\begin{lemma} \label{univ-ring}
Let \(R\) be a \(K\)-algebra with an idempotent Peirce decomposition. Then the ring \(\widetilde R = R \otimes_R R\) is a \(K\)-algebra with a firm Peirce decomposition \(\widetilde R_{ij} = R_{i*} \otimes_R R_{*j} \cong R_{ik} \otimes_{R_{kk}} R_{kj}\). The Peirce decomposition of \(R\) is firm if and only if \(R\) is a firm ring. The ideal \(I = \{x \in R \mid xR = Rx = 0\}\) is a \(K\)-submodule with a decomposition \(I = \bigoplus_{ij} I_{ij}\), where \(I_{ij} = I \cap R_{ij}\), and the induced Peirce decomposition on the factor-ring \(R / I\) is reduced.
\end{lemma}
\begin{proof}
Let us show that \(R_{ik} \otimes_{R_{kk}} R_{kj} \to \widetilde R_{ij}\) are isomorphisms for all \(i\), \(j\), \(k\). By lemma \ref{morita} applied to the ring \(\widetilde R\) and the modules \(P = R_{k*} \otimes_R R\), \(Q = R \otimes_R R_{*k}\) we get that the composition
\[R_{i*} \otimes_R R_{*k} \otimes_{R_{kk}} R_{k*} \otimes_R R_{*j}
\to R_{ik} \otimes_{R_{kk}} R_{kj}
\to \widetilde R_{ij}\]
is a bijection. The left map in this composition is surjective, so it is a bijection and the required map is also a bijection.

The ring \(\widetilde R\) is a \(K\)-algebra with \(k (x \otimes y) = kx \otimes y = x \otimes ky\), where the right equality follows from the idempotency of \(R\). The Peirce decomposition of \(\widetilde R\) is clearly firm since
\[(R_{ik} \otimes_{R_{kk}} R_{kj}) \otimes_{R_{jj} \otimes_{R_{jj}} R_{jj}} (R_{jl} \otimes R_{ll} R_{ls}) = R_{ik} \otimes R_{kk} R_{kj} \otimes_{R_{jj}} R_{jl} \otimes R_{ll} R_{ls} \cong R_{ij} \otimes R_{jj} R_{js},\]
so the first claim is proved. It follows that if \(R\) is firm, then \(\widetilde R \cong R\) and the Peirce decomposition of \(R\) is already firm. The converse follows from lemma \ref{root-elim}.

The last claim easily follows since if \(xR + Rx \leq I\), then \(xR^2 = R^2 x = 0\), i.e. \(Rx = xR = 0\) and \(x \in I\).
\end{proof}

Not all firm rings are reduced and not all reduced rings are firm. For example, let
\[A = \mathbb Z[x_0, x_1, x_2, \ldots] / \langle x_0 x_i,\, x_i x_0,\, x_i - x_{i + 1}^2 \rangle,\]
it is an idempotent commutative ring. It is easy to check that \(A\) is firm but not reduced since \(x_0 A = A x_0 = 0\) and \(x_0 \neq 0\). On the other hand, \(A / \mathbb Z x_0\) is a reduced commutative ring, but the canonical homomorphism from \((A / \mathbb Z x_0) \otimes_{A / \mathbb Z x_0} (A / \mathbb Z x_0) \cong A\) to \(A / \mathbb Z x_0\) is not an isomorphism. Similar examples may be constructed for Peirce decompositions of any rank \(\ell \geq 1\) using matrix algebras over \(A\) and \(A / \mathbb Z x_0\).

\section{Groups with commutator relations}

An element \(x \in R\) of a ring is called \textit{quasi-invertible} if there is \(y \in R\) such that \(xy + x + y = yx + x + y = 0\). In other words, \(x\) is quasi-invertible if \(x + 1\) is invertible in the ``unitalization'' \(R \rtimes \mathbb Z\) of \(R\). The set of quasi-invertible elements \(R^\circ\) of \(R\) is a group with respect to \(x \circ y = xy + x + y\). If \(R\) is unital, then there is a canonical isomorphism \(R^\circ \to R^*,\, x \mapsto x + 1\). The group \(R^\circ\) acts on \(R\) by automorphisms via \(\up x y = (xy + y) x^{\circ (-1)} + xy + y\), where \(x^{\circ (-1)}\) is the quasi-inverse of \(x \in R^\circ\).

Let \(R\) be a ring with a Peirce decomposition. We denote the group \(R^\circ\) by \(\glin(R)\) and call it the \textit{general linear group} of \(R\). For any \(i \neq j\) and \(a \in R_{ij}\) the \textit{elementary transvection} \(t_{ij}(a) = a\) lie in \(R^\circ\). They satisfy the \textit{Steinberg relations}
\begin{itemize}
\item \(t_{ij}(a) \circ t_{ij}(b) = t_{ij}(a + b)\);
\item \([t_{ij}(a), t_{kl}(b)]_\circ = 0\) for \(j \neq k
\) and \(i \neq l\);
\item \([t_{ij}(a), t_{jk}(b)]_\circ = t_{ik}(ab)\) for \(i \neq k\).
\end{itemize}
The \textit{elementary subgroup} \(\elin(R)\) is the subgroup of \(\glin(R)\) generated by the elementary transvections. The \textit{Steinberg group} \(\stlin(R)\) is the abstract group with the generators \(x_{ij}(a)\) for \(i \neq j\) and \(a \in R_{ij}\) and the Steinberg relations.

\begin{lemma} \label{center-perf}
If the Peirce decomposition of \(R\) is idempotent and of rank \(\ell \geq 3\), then \(\stlin(R)\) and \(\elin(R)\) are perfect groups. If the Peirce decomposition is reduced, then the group of upper triangular elements \(\prod_{i < j}^\circ t_{ij}(R_{ij})\) has trivial intersection with the center of \(\elin(R)\) and injectively maps to \(\Aut(R)\).
\end{lemma}
\begin{proof}
The first claim easily follows from the Steinberg relations. To prove the second claim, let \(g \in \prod_{i < j}^\circ t_{ij}(R_{ij})\) be in the center of \(\elin(R)\). In other words, \(g \in R\) has zero components in \(R_{ij}\) for \(i \geq j\) and it lies in the center of \(R\), since \(\bigcup_{i \neq j} R_{ij}\) generates the ring \(R\) for \(\ell \geq 2\) (in the case \(\ell \leq 1\) we already have \(g = 0\)). It follows that \(R_{ii} g_{ij} = g_{ij} R_{jj} = 0\) for \(i < j\), where \(g_{ij}\) is the component of \(g\) in \(R_{ij}\). Since the Peirce decomposition is reduced, we have \(R g_{ij} = g_{ij} R = 0\) for \(i < j\), that is \(g = 0\) as claimed. If an upper triangular \(g\) trivially acts on \(R\), then it also lies in the center of \(R\), so \(g = 0\) by the above argument.
\end{proof}

Let
\[\Phi = \{\mathrm e_i - \mathrm e_j \in \mathbb R^\ell \mid i \neq j\}\]
be the root system of type \(\mathsf A_{\ell - 1}\) for \(\ell \geq 1\). We say that a group \(G\) has \(\Phi\)-\textit{commutator relations} if there are \textit{root subgroups} \(U_\alpha \leq G\) for \(\alpha \in \Phi\) such that
\begin{itemize}
\item \([U_\alpha, U_\beta] = 1\) for \(\alpha + \beta \notin \Phi \cup \{0\}\);
\item \([U_\alpha, U_\beta] \leq U_{\alpha + \beta}\) for \(\alpha + \beta \in \Phi\).
\end{itemize}
See \cite[definition 3.2]{st-jordan} for a generalization to arbitrary sets of roots.

Recall the group-theoretic identities
\begin{align*}
[xy, z] &= \up x{[y, z]}\, [x, z]; \label{l} \tag{L} \\
[x, yz] &= [x, y]\, \up y{[x, z]}; \label{r} \tag{R} \\
\up y{[x, [y^{-1}, z]]}\, \up z{[y, [z^{-1}, x]]}\, \up x{[z, [x^{-1}, y]]} &= 1. \label{hw} \tag{HW}
\end{align*}

Let \(G\) be a group with \(\Phi\)-commutator relations. By (\ref{l}) and (\ref{r}) the maps
\[c_{\alpha \beta} \colon U_\alpha \times U_\beta \to U_{\alpha + \beta}, (x, y) \mapsto [x, y]\]
are biadditive for all bases \((\alpha, \beta)\) of root subsystems of type \(\mathsf A_2\) (i.e. the pairs of roots with the angle \(\frac{2 \pi}3\) between them), so we may consider them as homomorphisms \(U_\alpha \otimes U_\beta \to U_{\alpha + \beta}\). If \((\alpha, \beta, \gamma)\) is a basis of a root subsystem of type \(\mathsf A_3\) (such that \(\alpha \perp \gamma\)), then
\[[x, [y, z]] = [[x, y], z]\]
for all \(x \in U_\alpha\), \(y \in U_\beta\), \(z \in U_\gamma\) as a corollary from (\ref{hw}). Actually, these are the only relations for
\[U_{\alpha + \beta + \gamma} \rtimes U_{\beta + \gamma} \rtimes U_{\alpha + \beta} \rtimes U_\gamma \rtimes U_\beta \rtimes U_\alpha\]
to be a group.

Recall that up to the order the only bases of root subsystems of \(\Phi\) of type \(\mathsf A_2\) are the pairs \((\mathrm e_i - \mathrm e_j, \mathrm e_j - \mathrm e_k)\) for distinct \(i\), \(j\), \(k\), and the only bases of root subsystems of type \(\mathsf A_3\) are the triples \((\mathrm e_i - \mathrm e_j, \mathrm e_j - \mathrm e_k, \mathrm e_k - \mathrm e_l)\) for distinct \(i\), \(j\), \(k\), \(l\). We say that \(\Phi\)-commutator relations are
\begin{itemize}
\item \textit{idempotent}, if \([U_\alpha, U_\beta] = U_{\alpha + \beta}\) for any base \((\alpha, \beta)\) of a root subsystem of type \(\mathsf A_2\);
\item \textit{firm}, if they are idempotent and for every basis \((\alpha, \beta, \gamma)\) a root subsystem of type \(\mathsf A_3\) (such that \(\alpha \perp \gamma\)) the kernel of
\[(c_{\alpha, \beta + \gamma} \enskip c_{\alpha + \beta, \gamma}) \colon (U_\alpha \otimes U_{\beta + \gamma}) \oplus (U_{\alpha + \beta} \otimes U_\gamma) \to U_{\alpha + \beta + \gamma}\]
coincides with the image of
\[\bigl(\begin{smallmatrix} 1 \otimes c_{\beta \gamma} & c_{\alpha + \beta, -\beta} \otimes 1 \\ -c_{\alpha \beta} \otimes 1 & -1 \otimes c_{-\beta, \beta + \gamma} \end{smallmatrix}\bigr) \colon (U_\alpha \otimes U_\beta \otimes U_\gamma) \oplus (U_{\alpha + \beta} \otimes U_{-\beta} \otimes U_{\beta + \gamma}) \to (U_\alpha \otimes U_{\beta + \gamma}) \oplus (U_{\alpha + \beta} \otimes U_\gamma).\]
\item \textit{reduced}, if they are idempotent and for any root subsystem of type \(\mathsf A_2\) and any root \(\alpha\) from this subsystem there are no non-trivial \(g \in U_\alpha\) such that \([g, U_\beta] = 1\) for all \(\beta \neq -\alpha\) from this subsystem.
\end{itemize}

Informally, the idempotence condition says that every root subgroup may be expressed in terms of the other root subgroups with the roots from any fixed root subsystem of type \(\mathsf A_2\). The firmness condition says that the only relations between the generators \(c_{\beta, \alpha - \beta}(x, y)\) of \(U_\alpha\) are the biadditivity and the corollary of (\ref{hw}), but if we consider only the roots from any fixed root subsystem of type \(\mathsf A_3\). Finally, the reducibility condition says that the elements of the root subgroups are completely determined by their conjugacy actions on other root subgroups with the roots from any root subsystem of type \(\mathsf A_2\).

We also say that \(\Phi\)-commutator relations are \(K\)-\textit{linear} for a unital commutative ring \(K\) if the abelian groups \(U_\alpha\) have structures of unital \(K\)-modules and the maps \(c_{\alpha \beta}\) are \(K\)-bilinear.

\begin{lemma} \label{gl-roots}
Let \(R\) be a \(K\)-algebra with a Peirce decomposition of rank \(\ell \geq 1\). Then \(\glin(R)\), \(\elin(R)\), and \(\stlin(R)\) have \(K\)-linear \(\Phi\)-commutator relations with \(U_{\mathrm e_i - \mathrm e_j} = t_{ij}(R_{ij})\) or \(U_{\mathrm e_i - \mathrm e_j} = x_{ij}(R_{ij})\). If the Peirce decomposition is idempotent, firm, or reduced, then the resulting \(\Phi\)-commutator relations have the same property.
\end{lemma}
\begin{proof}
The only non-trivial claim is that the \(\Phi\)-commutator relations are firm if the Peirce decomposition is firm. We have to check that the kernel of
\[(m \enskip m) \colon (R_{ij} \otimes R_{jl}) \oplus (R_{ik} \otimes R_{kl}) \to R_{il}\]
coincides with the image of
\[\bigl(\begin{smallmatrix} 1 \otimes m & m \otimes 1 \\ -m \otimes 1 & -1 \otimes m \end{smallmatrix}\bigr) \colon (R_{ij} \otimes R_{jk} \otimes R_{kl}) \oplus (R_{ik} \otimes R_{kj} \otimes R_{jl}) \to (R_{ij} \otimes R_{jl}) \oplus (R_{ik} \otimes R_{kl}),\]
where \(m\) denote the multiplication homomorphisms. Notice that this image contains the images of
\begin{align*}
R_{ij} \otimes R_{jj} \otimes R_{jl} &\to R_{ij} \otimes R_{jl},\, x \otimes y \otimes z \mapsto xy \otimes z - x \otimes yz; \\
R_{ik} \otimes R_{kk} \otimes R_{kl} &\to R_{ik} \otimes R_{kl},\, x \otimes y \otimes z \mapsto xy \otimes z - x \otimes yz;
\end{align*}
since \(xyz \otimes w - x \otimes yzw = (xyz \otimes w - xy \otimes zw) + (xy \otimes zw - x \otimes yzw)\) for \(x \in R_{ij}\), \(y \in R_{jk}\), \(z \in R_{kj}\), \(w \in R_{jl}\) or \(x \in R_{ik}\), \(y \in R_{kj}\), \(z \in R_{jk}\), \(w \in R_{kl}\). Then the claim follows from lemma \ref{root-elim}.
\end{proof}

Let \(G\) be a group with firm or reduced \(K\)-linear \(\Phi\)-commutator relations. A \textit{coordinatization} of \(G\) is a group homomorphism \(\pi \colon \stlin(R) \to G\) inducing \(K\)-linear isomorphisms on the root subgroups, where \(R\) is a \(K\)-algebra with a Peirce decomposition of rank \(\ell\) and this Peirce decomposition is firm or reduced respectively. The next two lemmas show that there is at most one coordinatization in each case up to a unique isomorphism.

\begin{lemma} \label{firm-un}
Let \(\pi \colon \stlin(R) \to G\) be a coordinatization of a group with firm \(K\)-linear \(\Phi\)-commutator relations and \(\rho \colon \stlin(S) \to G\) be a group homomorphism inducing \(K\)-linear maps between the root subgroups, where \(S\) is a \(K\)-algebra with a firm Peirce decomposition of rank \(\ell \geq 3\). Then there is a unique \(K\)-algebra homomorphism \(f \colon S \to R\) preserving the Peirce decomposition and such that \(\rho = \pi \circ \stlin(f)\).
\end{lemma}
\begin{proof}
We have to construct the \(K\)-linear homomorphisms \(f_{ij} \colon S_{ij} \to R_{ij}\). If \(i \neq j\), then they are uniquely determined by the maps \(\pi\) and \(\rho\) and \(f_{ik}(xy) = f_{ij}(x)\, f_{jk}(y)\) for \(x \in S_{ij}\), \(y \in S_{jk}\), and distinct \(i\), \(j\), \(k\).

Let
\[f_{ii}^j \colon S_{ij} \otimes S_{ji} \to R_{ii}, x \otimes y \mapsto f_{ij}(x)\, f_{ji}(y)\]
for \(i \neq j\). Since
\[f_{ii}^j(x \otimes yz) = f_{ii}^k(xy \otimes z)\]
for \(x \in S_{ij}\), \(y \in S_{jk}\), \(z \in S_{ki}\), and distinct \(i\), \(j\), \(k\), we get
\[f_{ii}^j(x \otimes yzw) = f_{ii}^j(xyz \otimes w)\]
for \(x \in S_{ij}\), \(y \in S_{jk}\), \(z \in S_{kj}\), \(w \in S_{ji}\), and distinct \(i\), \(j\), \(k\). Using that the Peirce decomposition of \(S\) is firm we obtain that there are unique homomorphisms \(f_{ii} \colon S_{ii} \to R_{ii}\) such that \(f_{ii}^j(x \otimes y) = f_{ii}(xy)\). It is easy to check that the resulting map \(f = \bigoplus_{ij} f_{ij}\) is a \(K\)-algebra homomorphism. Clearly, it is unique.
\end{proof}

\begin{lemma} \label{red-un}
Let \(\pi \colon \stlin(R) \to G\) be a coordinatization of a group with reduced \(K\)-linear \(\Phi\)-commutator relations and \(\rho \colon \stlin(S) \to G\) be a group homomorphism inducing surjective \(K\)-linear maps between the root subgroups, where \(S\) is a \(K\)-algebra with an idempotent Peirce decomposition of rank \(\ell \geq 3\). Then there is a unique \(K\)-algebra homomorphism \(f \colon S \to R\) preserving the Peirce decomposition and such that \(\rho = \pi \circ \stlin(f)\), it is necessarily surjective.
\end{lemma}
\begin{proof}
As in the proof of lemma \ref{firm-un}, we have \(K\)-linear surjective homomorphisms \(f_{ij} \colon S_{ij} \to R_{ij}\) for \(i \neq j\), satisfying \(f_{ik}(xy) = f_{ij}(x)\, f_{jk}(y)\) for distinct \(i, j, k\). Let
\[f_{ii}^j \colon S_{ij} \otimes S_{ji} \to R_{ii}, x \otimes y \mapsto f_{ij}(x)\, f_{ji}(y)\]
for \(i \neq j\). Since \(\sum_{t \in T} x_t y_t = 0\) for \(x_t \in R_{ij}\), \(y_t \in R_{ji}\) if and only if \(\sum_{t \in T} x_t y_t z = 0\) for all \(z \in R_{ik}\) and \(\sum_{t \in T} w x_y y_t = 0\) for all \(w \in R_{ki}\), where \(i\), \(j\), \(k\) are distinct, and the Peirce decomposition of \(R\) is reduced, the maps \(f_{ii}^j\) factor through \(S_{ii}\). The resulting homomorphisms \(f_{ii} \colon S_{ii} \to R_{ii}\) are independent on \(j\) since \(f_{ii}^j(x \otimes yz) = f_{ii}^k(xy \otimes z)\) for \(x \in S_{ij}\), \(y \in S_{jk}\), \(z \in S_{ki}\), and distinct \(i\), \(j\), \(k\). It is easy to check that the map \(f = \bigoplus_{ij} f_{ij}\) is a \(K\)-algebra homomorphism. Clearly, it is unique and surjective.
\end{proof}

\section{Coordinatization theorem}

In this section we prove that that coordinatizations always exist for \(\ell \geq 4\). For convenience some parts of the proofs are given in separate lemmas.

\begin{lemma} \label{ass}
Let \(R = \bigoplus_{1 \leq i, j \leq \ell} R_{ij}\) be an abelian group with the multiplication homomorphisms \(R_{ij} \times R_{jk} \to R_{ik}\) for \(\ell \geq 4\). Suppose that \(R_{ij} R_{jk} = R_{ik}\) for \(i \neq j \neq k\) (but possibly \(i = k\)) and the associativity rule \((xy)z = x(yz)\) holds for \(x \in R_{ij}\), \(y \in R_{jk}\), \(z \in R_{kl}\), where the indices are distinct; \(i \neq j \neq k \neq l = i\); \(i \neq j \neq l \neq k = i\); \(j \neq i \neq k \neq l = j\); \(i = k = l \neq j\); or \(i = j = l \neq k\). Then \(R\) is an associative ring with an idempotent Peirce decomposition.
\end{lemma}
\begin{proof}
The associativity of \(R\) means that \((xy)z = x(yz)\) for \(x \in R_{ij}\), \(y \in R_{jk}\), \(z \in R_{kl}\). Depending on the coincidences between the indices, there are \(15\) cases and \(6\) of them are already known. The case \(i = j \neq k \neq l \neq j\) follows using
\[((xy)z)w = (x(yz))w = x((yz)w) = x(y(zw)) = (xy)(zw)\]
for \(x \in R_{is}\), \(y \in R_{si}\), \(z \in R_{ik}\), \(w \in R_{kl}\), and distinct \(i\), \(k\), \(l\), \(s\); the case \(i \neq j \neq k = l \neq i\) is symmetric to this. The remaining cases \(i = k \neq j = l\); \(i \neq j = k \neq l \neq i\); \(i = l \neq j = k\); \(i = j = k \neq l\); \(i \neq j = k = l\); \(i = j \neq k = l\); and \(i = j = k = l\) follow from the known cases using
\[(x(yz))w = ((xy)z)w = (xy)(zw) = x(y(zw)) = x((yz)w)\]
for \(x \in R_{ij}\), \(y \in R_{js}\), \(z \in R_{sk}\), \(w \in R_{kl}\), and \(s \notin \{i, j, k, l\}\).

Now suppose that \(i\), \(j\), \(k\) are indices and not all of them are distinct. Take distinct \(s, t \notin \{i, j, k\}\). Then
\[R_{ik} = R_{it} R_{ts} R_{sk} = R_{it} R_{tj} R_{js} R_{sk} = R_{ij} R_{jk}.\qedhere\]
\end{proof}

\begin{theorem} \label{firm-ex}
Let \(K\) be a unital commutative ring and \(G\) be a group with firm \(K\)-linear \(\Phi\)-commutator relations, where \(\Phi\) is a root system of type \(\mathsf A_{\ell - 1}\) for \(\ell \geq 4\). Then \(G\) admits a coordinatization, it is unique up to a unique isomorphism.
\end{theorem}
\begin{proof}
The uniqueness follows from lemma \ref{firm-un}. Let \(R_{ij} = U_{\mathrm e_i - \mathrm e_j}\) for \(i \neq j\) and \(m \colon R_{ij} \otimes R_{jk} \to R_{ik},\, x \otimes y \mapsto xy\) be the multiplication maps induced by the commutators. They are bilinear and \((xy)z = x(yz)\) for \(x \in R_{ij}\), \(y \in R_{jk}\), \(z \in R_{kl}\), and distinct \(i\), \(j\), \(k\), \(l\).

For distinct indices \(i\), \(j\), \(s\) let \(A_{sijs}\) be the image of
\[\bigl(\begin{smallmatrix} 1 \otimes m \\ -m \otimes 1 \end{smallmatrix}\bigr) \colon R_{si} \otimes R_{ij} \otimes R_{js} \to (R_{si} \otimes R_{is}) \oplus (R_{sj} \otimes R_{js})\]
and
\[R_{ss} = \bigoplus_{i \neq s} (R_{si} \otimes R_{is}) / \sum_{s \neq i \neq j \neq s} A_{sijs}.\]

Now we construct the multiplication on \(R\). By lemma \ref{r-cons} below and the identity
\[x((yz)w) = x(y(zw)) = (xy)(zw)\]
for \(x \in R_{ik}\), \(y \in R_{kl}\), \(z \in R_{li}\), \(w \in R_{ij}\) with distinct \(i\), \(j\), \(k\), \(l\) there are unique homomorphisms \(R_{ii} \otimes R_{ij} \to R_{ij}\) for all \(i \neq j\) such that the associativity holds for \(R_{ik} \otimes R_{ki} \otimes R_{ij} \to R_{ij}\) with distinct \(i\), \(j\), \(k\). Similarly, there are the multiplications \(R_{ij} \otimes R_{jj} \to R_{ij}\) for \(i \neq j\). Finally, by lemma \ref{r-cons} and the identity
\[x(y(zw)) = x((yz)w) = (x(yz))w = ((xy)z)w\]
for \(x \in R_{ij}\), \(y \in R_{ji}\), \(z \in R_{ik}\), \(w \in R_{ki}\) with distinct \(i\), \(j\), \(k\) there are unique homomorphisms \(R_{ii} \otimes R_{ii} \to R_{ii}\) such that the associativity holds for \(R_{ij} \otimes R_{ji} \otimes R_{ii} \to R_{ii}\) and \(R_{ii} \otimes R_{ij} \otimes R_{ji} \to R_{ii}\) with \(i \neq j\). The resulting multiplication is associative by lemma \ref{ass}, i.e. \(R\) is a ring with an idempotent Peirce decomposition.

The ring \(R\) is firm since the \(\Phi\)-commutator relations are firm and lemma \ref{r-cons} holds. Then the Peirce decomposition of \(R\) is also firm by lemma \ref{univ-ring}. Since \(R_{ii} \cong R_{ij} \otimes_{R_{jj}} R_{ji}\), \(R_{ij} = R_{ij} R_{jk} R_{kj}\), and \(R_{ji} = R_{jk} R_{kj} R_{ji}\) for distinct \(i\), \(j\), \(k\), there is a unique \(K\)-module structure on \(R_{ii}\) such that the multiplication \(R_{ij} \times R_{ji} \to R_{ii}\) is bilinear. Moreover, it is independent on \(j\) since \(R_{ii} = R_{ij} R_{jk} R_{ki}\) for distinct \(i\), \(j\), \(k\). The multiplication map \(R_{ii} \times R_{ij} \to R_{ij}\) is \(K\)-bilinear since \(R_{ii} = R_{ik} R_{ki}\) for all distinct \(i\), \(j\), \(k\), and the multiplication \(R_{ij} \times R_{jj} \to R_{ij}\) is \(K\)-bilinear by the symmetry. Finally, the multiplication \(R_{ii} \times R_{ii} \to R_{ii}\) is \(K\)-bilinear since \(R_{ii} = R_{ij} R_{ji}\) for all \(i \neq j\).
\end{proof}

\begin{lemma} \label{r-cons}
The canonical maps \(\bigl((R_{si} \otimes R_{is}) \oplus (R_{sj} \otimes R_{js})\bigr) / (A_{sijs} + A_{sjis}) \to R_{ss}\) are bijective for distinct \(i\), \(j\), \(s\).
\end{lemma}
\begin{proof}
Fix distinct indices \(s\), \(i\), \(j\). The identity
\[x \otimes (yz)w - x(yz) \otimes w = (x \otimes y(zw) - xy \otimes zw) + (xy \otimes zw - (xy)z \otimes w)\]
for \(x \in R_{sp}\), \(y \in R_{pr}\), \(z \in R_{rq}\), \(w \in R_{qs}\) implies that \(A_{spqs} \leq A_{sprs} + A_{srqs}\) and the identity
\[xy \otimes zw - (xy)z \otimes w = (xy \otimes zw - x \otimes y(zw)) - (x(yz) \otimes w - x \otimes (yz)w)\]
for \(x \in R_{sr}\), \(y \in R_{rp}\), \(z \in R_{pq}\), \(w \in R_{qs}\) implies that \(A_{spqs} \leq A_{srps} + A_{srqs}\) for distinct \(s\), \(p\), \(q\), \(r\). It follows that
\[\sum_{s \neq p \neq q \neq s} A_{spqs} = \sum_{s \neq p \neq i} (A_{sips} + A_{spis}) = A_{sijs} + A_{sjis} + \sum_{p \notin \{s, i, j\}} A_{sjps}.\]

From \(R_{sp} = R_{sj} R_{jp}\) we obtain \(R_{sp} \otimes R_{ps} \leq R_{sj} \otimes R_{js} + A_{sjps}\). It remains to show that
\[A_{sjps} \cap (R_{sj} \otimes R_{js}) \leq A_{sijs} + A_{sjis}\]
for \(p \notin \{s, i, j\}\). Indeed, let \(\sum_{t \in T} a_t \otimes b_t c_t \in R_{sj} \otimes R_{js}\) be such that \(a_t \in R_{sj}\), \(b_t \in R_{jp}\), \(c_t \in R_{ps}\), and \(\sum_{t \in T} a_t b_t \otimes c_t = 0\). Since the \(\Phi\)-commutator relations of \(G\) are firm and right exact sequences are preserved under tensor products, there are \(x_h \in R_{si}\), \(y_h \in R_{ij}\), \(z_h \in R_{jp}\), \(w_h \in R_{ps}\), \(x'_{h'} \in R_{sj}\), \(y'_{h'} \in R_{ji}\), \(z'_{h'} \in R_{ip}\), \(w'_{h'} \in R_{ps}\) such that
\[\sum_{t \in T} a_t \otimes b_t \otimes c_t = \sum_{h \in H} (x_h y_h \otimes z_h \otimes w_h - x_h \otimes y_h z_h \otimes w_h) + \sum_{h' \in H'} (x'_{h'} \otimes y'_{h'} z'_{h'} \otimes w'_{h'} - x'_{h'} y'_{h'} \otimes z'_{h'} \otimes w'_{h'})\]
It follows that \(\sum_{t \in T} a_t \otimes b_t c_t \in A_{sijs} + A_{sjis}\).
\end{proof}

\begin{theorem} \label{red-ex}
Let \(K\) be a unital commutative ring and \(G\) be a group with reduced \(K\)-linear \(\Phi\)-commutator relations, where \(\Phi\) is a root system of type \(\mathsf A_{\ell - 1}\) for \(\ell \geq 4\). Then \(G\) admits a coordinatization, it is unique up to a unique isomorphism.
\end{theorem}
\begin{proof}
The uniqueness follows from lemma \ref{red-un}. Let \(R_{ij} = U_{\mathrm e_i - \mathrm e_j}\) for \(i \neq j\) and \(m \colon R_{ij} \otimes R_{jk} \to R_{ik},\, x \otimes y \mapsto xy\) be the multiplication maps induced by the commutators. They are bilinear and \((xy)z = x(yz)\) for \(x \in R_{ij}\), \(y \in R_{jk}\), \(z \in R_{kl}\), and distinct \(i\), \(j\), \(k\), \(l\).

Fix an index \(s\). We construct the ring \(R_{ss}\) as a subring of \(E = \prod_{i \neq s} (\End(R_{is})^\op \times \End(R_{si}))\), where \(\End\) denote the rings of endomorphisms of \(K\)-modules. If \(s \neq i \neq j \neq s\), \(x \in R_{si}\), \(y \in R_{ij}\), \(z \in R_{js}\), then let \(\langle x, y, z \rangle_{ij} \in E\) be the element with the components
\begin{align*}
w &\mapsto x((yz)w) \text{ for } w \in R_{sk},\, k \notin \{i, s\}; &
w &\mapsto (xy)(zw) \text{ for } w \in R_{sk},\, k \notin \{j, s\}; \\
w &\mapsto (wx)(yz) \text{ for } w \in R_{ks},\, k \notin \{i, s\}; &
w &\mapsto (w(xy))z \text{ for } w \in R_{ks},\, k \notin \{j, s\}.
\end{align*}
The subring \(R_{ss} \subseteq E\) is generated by all \(\langle x, y, z \rangle_{ij}\). By lemma \ref{r-gen} below it is generated by \(\langle x, y, z \rangle_{ij}\) for any fixed \(i\), \(j\) and the homomorphism \(R_{ss} \to \End(R_{is})^\op \times \End(R_{si})\) is injective for any \(i \neq s\).

Let us construct the multiplication on \(R\). It is easy to see that there are unique maps \(R_{si} \times R_{is} \to R_{ss}\) for \(i \neq s\) such that \(\langle x, y, z \rangle_{ij} = x(yz) = (xy)z\). The multiplication maps \(R_{ss} \times R_{si} \to R_{si}\) and \(R_{is} \times R_{ss} \to R_{is}\) for \(i \neq s\) are given directly by the embedding of \(R_{ss}\) into \(E\), they clearly satisfy \((xy)z = x(yz)\) for \(x \in R_{sj}\), \(y \in R_{js}\), \(z \in R_{si}\) or \(x \in R_{is}\), \(y \in R_{sj}\), \(z \in R_{js}\), where \(i\), \(j\), \(s\) are distinct.

Finally, \(R_{ss}\) is a subring of \(E\) and \((xy)z = x(yz)\) for \(x \in R_{si}\), \(y \in R_{is}\), \(z \in R_{ss}\) or \(x \in R_{ss}\), \(y \in R_{si}\), \(z \in R_{is}\), where \(i \neq s\). Indeed, if \(s\), \(i\), \(j\), \(k\) are distinct, \(x \in R_{si}\), \(y \in R_{is}\), \(z \in R_{sj}\), \(w \in R_{js}\), \(u \in R_{sk}\), then
\begin{align*}
(((xy)z)w)u &= (x(yz))(wu) = x((yz)(wu)), \\
(xy)((zw)u) &= x(y(z(wu))) = x((yz)(wu)), \\
(x(y(zw)))u &= x(((yz)w)u) = x((yz)(wu)),
\end{align*}
and the symmetric identities also hold. From lemma \ref{ass} it follows that \(R\) is an associative ring and its Peirce decomposition is idempotent. Clearly, this is a reduced Peirce decomposition of a \(K\)-algebra.
\end{proof}

\begin{lemma} \label{r-gen}
The group \(R_{ss}\) is generated by \(\langle x, y, z \rangle_{ij}\) for any fixed \(i\), \(j\) and the homomorphism \(R_{ss} \to \End(R_{is})^\op \times \End(R_{si})\) is injective for any \(i\).
\end{lemma}
\begin{proof}
Let \(s\), \(i\), \(j\), \(k\) be different indices. Then \(\langle x, y, zw \rangle_{ij} = \langle x, yz, w \rangle_{ik}\) for \(x \in R_{si}\), \(y \in R_{ij}\), \(z \in R_{jk}\), \(w \in R_{ks}\) since
\begin{align*}
x((y(zw))u) &= x(((yz)w)u) \text{ if } l \neq i; \\
(xy)((zw)u) = (xy)(z(wu)) &= ((xy)z)(wu) = (x(yz))(wu) \text{ if } l = i; \\
(vx)(y(zw)) &= (vx)((yz)w) \text{ if } l \neq i; \\
(v(xy))(zw) = ((v(xy))z)w &= (v((xy)z))w = (v(x(yz)))w \text{ if } l = i
\end{align*}
for \(u \in R_{sl}\) and \(v \in R_{ls}\). Also, \(\langle x, yz, w \rangle_{ik} = \langle xy, z, w \rangle_{jk}\) for \(x \in R_{si}\), \(y \in R_{ij}\), 
\(z \in R_{jk}\), \(w \in R_{ks}\) by the symmetry. This implies the first claim.

In order to prove the second claim it suffices to show that if \(\sum_{t \in T} \langle x_t, y_t, z_t \rangle_{ij}\) has trivial image in \(\End(R_{sj})\), then it has trivial image in \(\End(R_{sk})\) for \(k \notin \{i, j, s\}\) (a similar result for the opposite endomorphism rings follows by the symmetry). Indeed, if \(u \in R_{sj}\), \(v \in R_{jk}\) are any elements, then
\[\sum_{t \in T} x_t ((y_t z_t) (uv)) = \sum_{t \in T} x_t (((y_t z_t) u) v) = \sum_{t \in T} (x_t ((y_t z_t) u)) v = 0. \qedhere\]
\end{proof}

\bibliographystyle{plain}  
\bibliography{references}

\end{document}